\documentclass[10pt]{amsart}
\usepackage{gensymb}
\usepackage{amscd}
\usepackage[all]{xy}
\usepackage{cleveref}
\usepackage{graphicx}
\usepackage{enumerate}
\def\reals{\mathbb{R}}
\def\complexes{\mathbb{C}}
\theoremstyle{theorem}
\newtheorem{theorem}{Theorem}[section]
\newtheorem{proposition}[theorem]{Proposition}

\theoremstyle{definition}

\makeindex

\begin{document}

\title[On the Lambert conformal conical projection]{On the Lambert conformal conical projection and the general map of the Russian Empire}
\author[Miyachi, Ohshika, Papadopoulos and Yamada]{Hideki Miyachi, Ken'ichi Ohshika, Athanase Papadopoulos and Sumio Yamada}
\address{Hideki Miyachi,  
School of Mathematics and Physics,
College of Science and Engineering,
Kanazawa University,
Kakuma-machi, Kanazawa,
Ishikawa, 920-1192, Japan}
\email{miyachi@se.kanazawa-u.ac.jp}  
\address{Ken'ichi Ohshika,
Department of Mathematics,
Gakushuin University,
Mejiro, Toshima-ku, Tokyo, Japan}
  \email{ohshika@math.gakushuin.ac.jp}
\address{Athanase Papadopoulos,
Institut de Recherche Mathématique Avancée
(Universitéde Strasbourg et CNRS),
7 rue René Descartes,
67084 Strasbourg Cedex France}
  \email{papadop@math.unistra.fr}
  \address{Sumio Yamada,
Department of Mathematics,
Gakushuin University,
Mejiro, Toshima-ku, Tokyo, Japan}
  \email{sumio.yamada@gmail.com}

\date{}
\maketitle

\begin{abstract}
 The problem of drawing geographical maps is the one of mapping a subset of the sphere, representing a country or some other region on the surface of the Earth, into the Euclidean plane, minimising certain distortion properties that are specified in advance. It is known that from the purely mathematical point of view, this is an extremely  difficult problem.
One of Leonhard Euler's duties during his first stay at the Imperial Academy of Sciences of Saint Petersburg (1727-1741) was to help establishing maps of the Russian Empire. He worked on this project under the direction of the famous French geographer Joseph-Nicolas Delisle, who was the head of the astronomy and geography departments of the Academy. 
The general map of the Russian Empire, together with several maps of its particular regions were published under Euler's direction  in the so-called Russian Atlas in 1745.
In his later memoir \emph{De proiectione geographica De Lisliana in mappa
generali imperii russici usitata}, written in 1777, Euler developed the mathematical  theory of the method used by Delisle on a heuristic basis, which he  himself used for drawing the general map of the Russian Empire. This method usually carries now the name Delisle--Euler map.
In a previous paper, the first two authors of the present paper compared the Delisle--Euler map with several other maps of the conical type, with respect to various mathematical distorsion properties. They showed that this map is the best one from all the points of view considered, when it is applied to the drawing of the Russian Empire. In the present paper, we compare the Euler--Delisle map with a map which was not considered in the paper mentioned, namely, the so-called Lambert conformal conical projection, applied to the same region of the Earth. We show that the latter is better in several respects than all the other maps considered in the previous paper, including the Delisle--Euler map.

 \bigskip


\bigskip


 \end{abstract}

\vfill\eject

\section{Introduction} 

The mathematical problem of drawing geographical maps is the problem of finding a mapping from a subset of the sphere, representing a country or some other
region on the surface of the Earth, into the Euclidean plane which satisfies certain
properties that are specified in advance. Examples of properties which are commonly
used in cartography are:

\begin{enumerate}
\item  conformality, that is, preservation of angles;

\item area preservation;

\item sending the parallels on the sphere to parallel lines in the plane;
 
\item sending the parallels to arcs of concentric circles;
 
\item sending the meridians of the sphere to parallel lines in the plane;
 
  \item sending the meridians to lines meeting at a common point;

  \item sending the meridians of the sphere to (pieces of) concentric circles or ellipses; and

\item \label{P6} having least distance distortion in a sense that has to be made precise (involving the bi-Lipschitz constant, the supremum of the local bi-Lipschitz constant, the quasi-conformal dilatation, etc.).
  \end{enumerate} 

Many other properties are also used in cartography. In general, a choice of more than one of the above properties is made for drawing a specific map. The properties desired depend on the use which  one wants to make of the map, and also on the shape of the region of the earth which has to be represented.
For instance, in one of the maps known usually\index{geographical map!Ptolemy} under the name\index{Ptolemy map} Ptolemy\index{Marinus--Ptolemy map} or Marinus--Ptolemy map,\footnote{The map is called after the geographer and astronomer Marinus  of Tyre\index{Marinus of Tyre} (c. 70-130 AD) and the mathematician, geographer and astronomer Claudius Ptolemy\index{Ptolemy (Claudius Ptolemy)} (c. 100-168 AD), also known as Ptolemy of  Alexandria.} as well as in the Delisle--Euler map\footnote{The map is called after the French astronomer and geographer Joseph-Nicolas Delisle\index{Delisle, Joseph-Nicolas} (1688-1768) and Leonhard Euler\index{Euler, Leonhard} (1707-1783), who worked at the same time at the Imperial Academy of Sciences of Saint Petersburg. Delisle was invited by the Tsar Peter the Great to realise his project of establishing maps of the newly founded Russian Empire: a general map, and maps of smaller regions. In 1736, Delisle started a huge programme of triangulating the Empire, a crucial step for the computation of  distances and coordinates, that is, the data used in creating the maps. We also note that Euler, besides being the outstanding mathematician as we know, was also, among others, a geographer and astronomer. During several years, he assisted Delisle, who was the head of the geography and astronomy departments of the Academy. Later on, after a conflict between Delisle and the Academy's administration, Euler became himself the head of that department.} which we shall mention thoroughly in this paper, the meridians are represented by straight lines in the plane which meet at a common point and the parallels are represented by concentric circles which are centred at that point.  Furthermore, there is a condition on the distance distortion being small on two chosen parallels. (In the case of Ptolemy, remarkable parallels, for the drawing of a map of the known world,  are  the parallel passing through the island of Thule, which was the most northerly  known location, 
the equator, which was the southern extreme, and the parallel passing through the island
of Rhodes, along which a large number of longitude measures have been conducted in Greek antiquity, see \cite[Chapter XXI]{Ptolemy-geo}.)

  Regarding Property (\ref{P6}) about distortion, we note the following, which shows that there is no map from a subset of the sphere to the Euclidean plane not distorting distances up to scale:
  
  \begin{proposition}\label{prop:positive}
There is no mapping for  a subset with non-empty interior of the sphere which preserves distances up to scale.
  \end{proposition} 
   This proposition was known since Greek Antiquity; it follows easily from propositions contained in the \emph{Spherics} of Menelaus of Alexandria, see \cite{RP2} and the report in \cite{Menelaus-Monthly}. For instance, it is an immediate consequence of Proposition 27 of the \emph{Spherics} which says that if $ABC$ is an arbitrary geodesic triangle on the sphere and if $D$ and $E$ are the midpoints of $AB$ and $BC$ respectively, then we have the following inequality between distances: $DE>AC/2$  (the equality is known today as a comparison inequality in positively curved metric spaces).  Proposition \ref{prop:positive} also follows (although in a slightly less obvious way)  from Proposition 12 of the \emph{Spherics} which says that in any spherical triangle the angle sum is strictly greater than two right angles.
 
In the paper \cite{MO}, the first two authors of the present paper compared a certain number of geographical maps\index{conical geographic projection} used for\index{geographical map!conical} the representation of the 18th century Russian Empire.  The reason for choosing  the Russian Empire is that Leonhard Euler, in his memoir
\cite{Euler-Delisle}, developed a theory for drawing geographical maps and he used it in the practical problem of drawing the general geographical maps of the Empire. We recall here that the drawing a general map of the Russian Empire and maps of  its specific regions were part of Euler's duties during his first stay at the  Imperial Academy of Sciences of Saint Petersburg (1727-1741). For the mathematical side of this undertaking, the interested reader may refer to Euler's memoirs on geography, written in 1777 and published the following year, \cite{Euler-De-represenatione, Euler2, Euler-Delisle}, all translated into English in the book \cite{CP}, especially the last one, which is concerned with the so-called Delisle--Euler map. The title of the last memoir is \emph{De proiectione geographica De Lisliana in mappa
generali imperii russici usitata}, whose translation is: ``On Delisle’s geographical projection used for a
general map of the Russian Empire". We also refer to the so-called \emph{Russian Atlas} published in Saint Petersburg in 1745 under Euler's direction \cite{atlas-Russicus}. In the introduction to that atlas, the Delisle--Euler method is outlined in a non-technical way. For a modern exposition of Euler's work on the Delisle map, see \cite{Geograp3, Geograpy4}. For a survey on Euler's work on geography and his relation with Delisle, the reader can refer to \cite{Geography-GB}.
The Delisle\index{Delise--Euler map} geographical map to which Euler referred in his memoir \cite{Euler-Delisle} is also called the Delisle--Euler map, a terminology we shall use here.

In the paper \cite{MO}, the authors  showed that the Delisle--Euler \index{Delise--Euler map} geographical map is the best, from various points of view which they made precise, among a set of geographical maps which they termed ``conical", used for the Russian Empire. The adjective ``conical" stems from the fact that these maps are first drawn on a cone whose apex is placed on the North-South axis of the sphere, and then, the cone is cut along one meridian and developed on the plane so that the final result is a map  drawn on the plane.\footnote{The reader should be aware of the fact that, unlike in mathematics,  there is no universally adopted nomenclature in geography, and in particular for what concerns geographical maps. The same maps have different names in different geographical tables, some maps are very poorly described, etc. The well-known geographer and geodesist  L. P. Lee, in the conclusion to his report on \emph{The nomenclature and classification of map projection} \cite{Lee}, writes: 
``[. . . ]
The foregoing observations reveal that the terminology of map projections is in a state of
confusion that would not be tolerated in any other modern science, and a systematization
is long overdue." Sometimes, historians of mathematics add to the confusion. For instance, we read at the beginning of the paper \cite{Abgrall} that the stereographic projection is a conical projection.} Among the conical projections, the Delisle--Euler map is of the type ``equidistant conical projection", in which the images of equidistant parallels are equidistant curves. 
The Delisle--Euler, unlike the Lambert conformal conical projection,\index{Lambert conformal conical projection} is\index{geographical map!Lambert conformal conical}  not conformal. The Delisle--Euler map is represented in Figure \ref{fig:Empire}.

\begin{figure}
\centering
\includegraphics[width=1\linewidth]{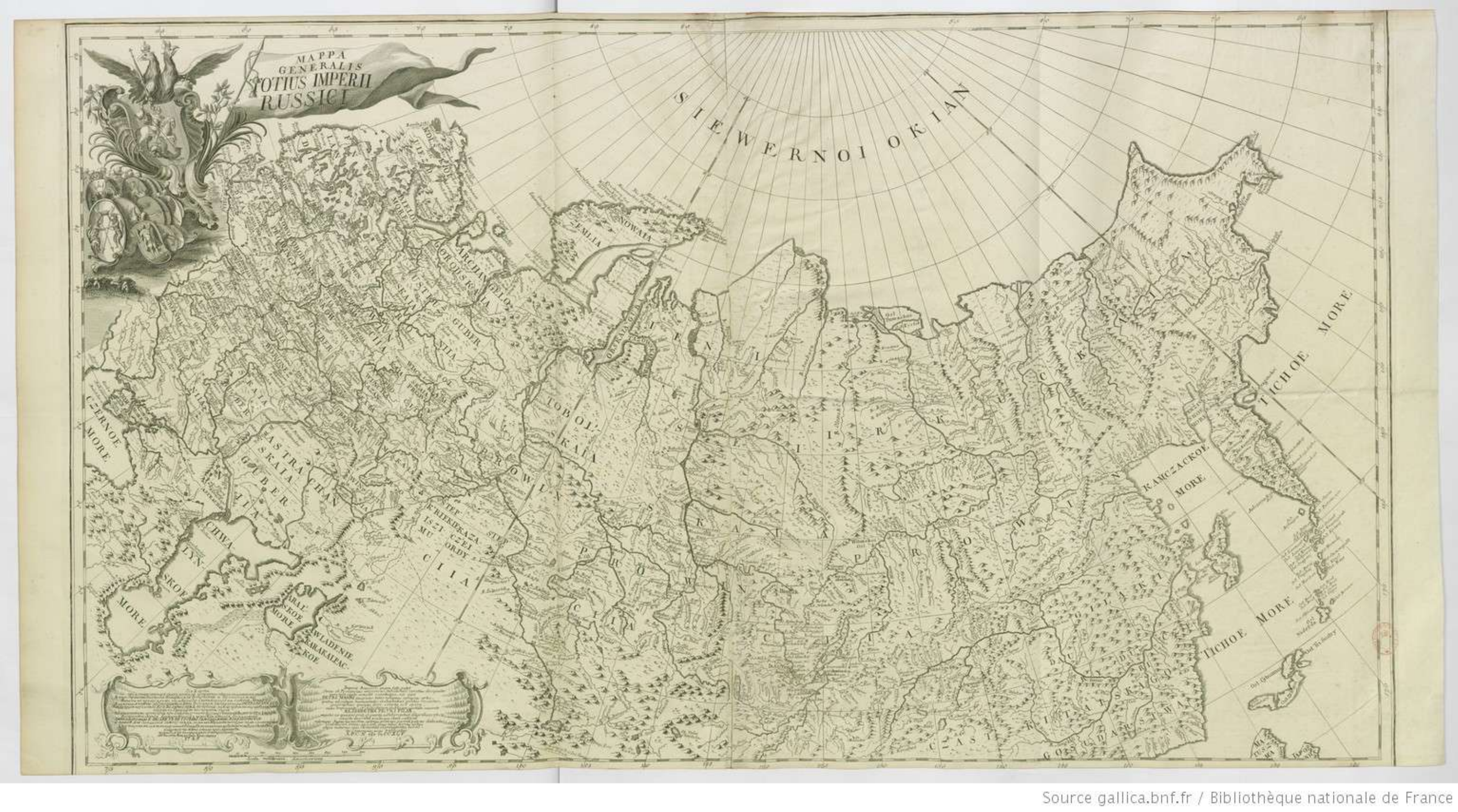}
\caption{An example of a projection map: Euler's map of the Russian Empire, from the \emph{Atlas Russicus} \cite{atlas-Russicus}, published in Saint Petersburg in 1745. Source: Gallica, Bibliothèque nationale de France.}
  \label{fig:Empire}
\end{figure}

Another illustration of a conical map projection is given in Figure \ref{fig:Conical-Ruysch}. This is a rendition of a map by the Dutch cartographer, astronomer and painter Johann Ruysch (c.\ 1460-1533) drawn in 1507. The map is titled \emph{Uniuersalior cogniti orbis tabula ex recentibus confecta obseruationibus} (A More Universal Map of the Known World Produced     from Recent Observations).
 Ruysch's map was included as a supplementary modern map to  the edition of Ptolemy's \emph{Geographia}  published in Rome in 1507. It is the second oldest known printed representation of the New World and it constitutes a significant  transition between the fifteen-century representations based on Ptolemy's views, and the new representations  made  after the discovery of the American continent. The representation of the West Indies on this map also takes into account the voyages of Christophorus Columbus.  This conical projection satisfies the further condition that at the equator, the length of a degree of latitude is equal to that of a degree of longitude.

\begin{figure}
\centering
\includegraphics[width=1\linewidth]{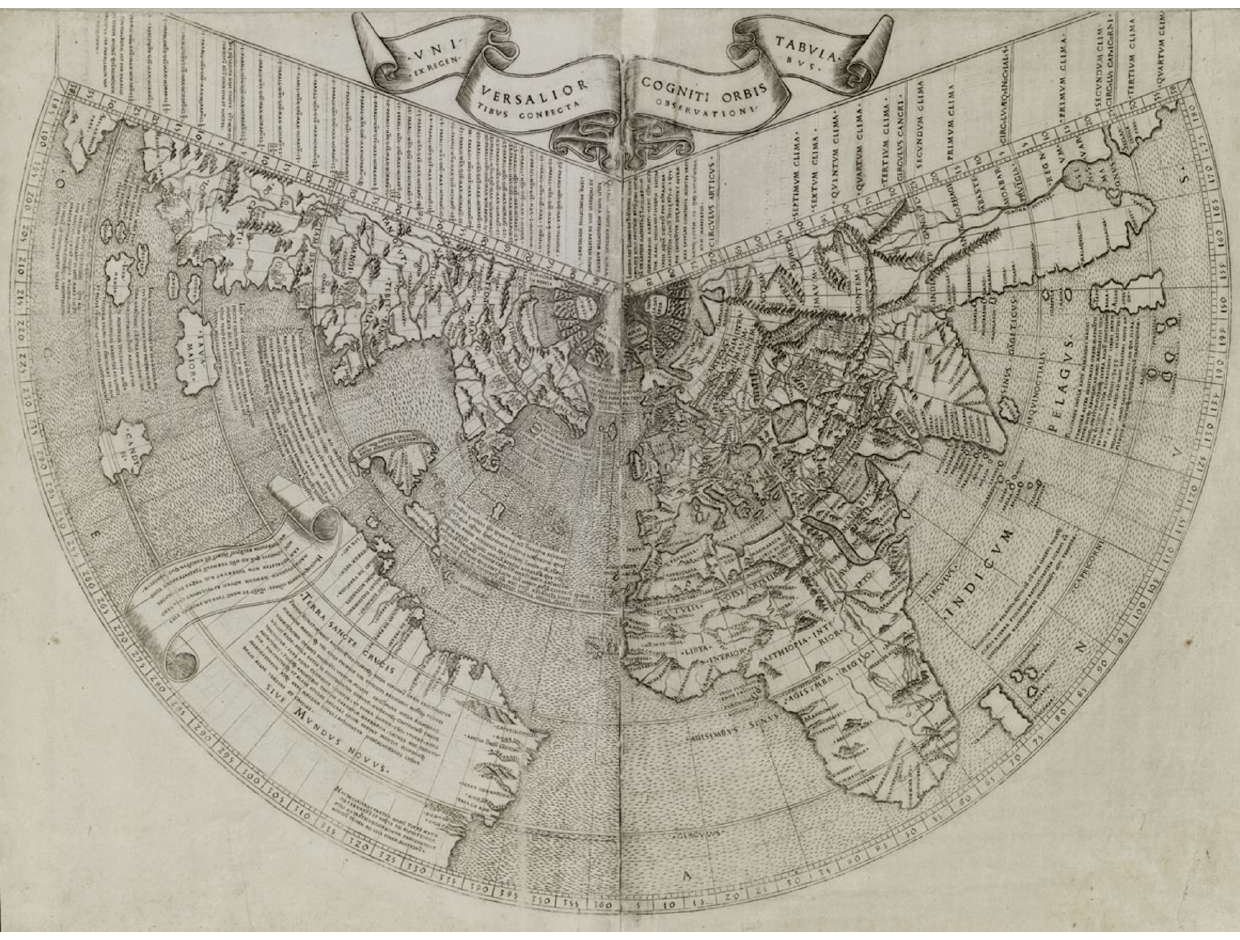}
\caption{Another example of a conical projection: A rendition of the map  \emph{Universalior Cogniti Orbis Tabula ex Recentibus Confecta Observationibus} by Johann Ruysch. Publisher  Bernardinus Venetus de Vitalibus. From the William C. Wonders Map Collection,  the University of Alberta Library.}
  \label{fig:Conical-Ruysch}
\end{figure}

The maps considered in the papers  \cite{MO} and  \cite{Geography-GB-OMP} carry the names
conical horizontal projection, conical orthogonal projections,\index{geographical map!Delisle--Euler} Delisle--Euler projection,\index{geographical map!Delisle--Euler}  Teichm\"uller projection and orthogonal projection. We also note that 
the\index{geographical map!of the Russian Empire (Delisle--Euler)} Delisle--Euler\index{General map of the Russian Empire (Delisle--Euler)}  method\index{geographical map!Delisle--Euler} turns out to be very close to the so-called Marinus--Ptolemy\index{Marinus--Ptolemy map}  map\index{geographical map!Marinus--Ptolemy} described by Ptolemy in the first-second centuries AD, see the paper \cite{Geography-GB-OMP}.

 The so-called Lambert conformal conical projection is a well-known method for constructing geographical maps, and it was often used in cartography until the 20th century; see \cite{Dav, Germain, Tissot}. One of the properties which distinguishes it from the other maps studied in the papers  \cite{MO} and  \cite{Geography-GB-OMP} is that it is conformal.  It was used in France for military purposes during World Wars I and II, see e.g. \cite{Dhollander, Stanley}.  This projection was not taken into account in the paper \cite{MO} where several other geographical maps were compared. 
 In the present paper, we compare the Lambert projection with the maps considered in  \cite{MO}, and we show that it is better than them in several respects.

\section{The Lambert conical projection}
A Lambert projection is a conformal map  from a spherical annulus to an annulus centred on the apex of a circular cone.
In this projection, meridians are represented by lines intersecting in a common point and parallels are represented by arcs of circles centred at this same point.
Let us describe it  more mathematically introducing some notation.
Consider a unit sphere $\mathbb{S}^2$ in the Euclidean space $\mathbb{R}^3$  with the north pole at $(0,0,1)$ and the south pole at $(0,0,-1)$,
and an annulus $A$ on $\mathbb{S}^2$ bounded by two parallels.
Consider a circular cone $C(\alpha)$ with cone point on the axis joining the two poles and situated above the north pole, having apex angle $2 \alpha$ and intersecting $\mathbb{S}^2$ along two circles centred at the north pole (see Figure \ref{fig:cone}.) These circles are called parallels.
Let  $L_0$ be the one of these parallels.
We consider a conformal map from the spherical annulus onto a conical annulus on $C(\alpha)$ bounded by two concentric circles.
Below, we shall explain how we normalise this map.
This map is evidently equivariant with respect to the rotation around the $z$-axis.
This is what we call a Lambert conical projection\index{Lambert conformal conical projection} with\index{geographical map!Lambert conformal conical} angle $\alpha$ (see Figure \ref{fig:cone1}).

Consider the angle parameter $\alpha$, measuring the angle between the vertical axis and a generator of the cone.
As $\alpha$ tends to $0$ while keeping one of the two parallels, say $L_0$,  intersecting the sphere, the Lambert projection with angle $\alpha$ converges to a so-called Mercator projection to the cylinder intersecting the sphere along $L_0$.\footnote{The Mercator projection or Mercator nautical map  is a conformal map of the type ``cylindrical projection", and sometimes the cylinder is assumed to be tangent to the equator. One must be aware of the fact that the geographical maps have sometimes different names in different treatises or catalogues on maps. This is why when we deal with a projection, from the mathematical point of view, we must define it precisely. We also note that the term projection is used here in a broader sense than it is used in the classical projective geometry.}
In the other direction, as $\alpha$ tends to $\pi/2$, keeping again the same parallel $L_0$, intersecting the sphere, the Lambert projection with angle $\alpha$ converges to the stereographic projection into the plane intersecting $\mathbb{S}^2$ along $L_0$.
 The existence of this family of Lambert projections connecting the Mercator\index{geographical map!Mercator} projection\index{Mercator projection (Mercator nautical map)}, with the stereographic projection\index{stereographic projection} was already observed by Lambert himself, see \S47 in \cite{LamG}. We reproduce here an excerpt, in the translation by Annette A'Campo in \cite{CP}, where Lambert considers this one-parameter family of maps:
 \begin{quote}\small
 
 The stereographic projection of the sphere as well as Mercator’s nautical map 
have the special property that all angles keep the values they have on the sphere. This
yields the greatest possible similarity that a plane figure can have with a figure drawn
on a sphere. But the question remained whether this property only holds for these
two kinds of projections or whether even though they seem to be very different,
they might be joined by intermediate forms of projections. Mercator represents
the meridians by parallel lines that intersect the equator orthogonally and that are
divided by the logarithms of the cotangents of half the equator height. The equator
 itself is divided into 360 equal parts as degrees. Hence in this projection the angle
of intersection of the meridians equals zero, since they are parallel. However, for
the stereographic projection from the pole, the meridians that are again represented
as lines, intersect under their true angles. Thus if between these
two kinds of projections there were intermediate forms, they should be looked for
by increasing or decreasing the angle of intersection of the meridians in arbitrary
proportion to their angle of intersection on the sphere. This is the way in which I
will proceed here.
 \end{quote}

It is conceivable that Lambert was the first in history to propose a one-parameter family of geographical map joining two significant geographical maps.

We shall compute distortion parameters of this family of projections (we shall say more precisely what distortions means). We first make a few remarks.
 
A Lambert projection, as we defined it, can be extended by analytic continuation to a map whose domain is the twice-punctured sphere $\mathbb{S}^2 \setminus \{(0,0,-1), (0,0,1)\}$.
Therefore, to normalise the projection, we shall instead specify a conformal map from $\mathbb{S}^2 \setminus \{(0,0,-1), (0,0,1)\}$ to the complement of the apex $P$ in $C(\alpha)$.
We choose the circular cone $C(\alpha)$ so that $P$ lies on the $z$-axis above the point $(0,0,1)$ and we call $L_0$ the lower parallel among the two parallels which are the intersection of  $\mathbb{S}^2$ with the cone. Furthermore, we call $\rho_0$ the height (that is, the $z$-coordinate) of the parallel $L_0$. We introduce a conformal map from $\mathbb{S}^2 \setminus \{(0,0,-1), (0,0,1)\}$ onto  $C(\alpha) \setminus \{P\}$ which is the identity on $L_0$.
We denote the $z$-coordinate of $L_0$ by $\rho_0$. The conical annulus $C(\alpha)$ in $\reals^3$ is also denoted by by $C(\alpha, \rho_0)$.
Note that the conformal map is uniquely determined, and can be expressed as a composition of the stereographic projection $\mathbb{S}^2 \setminus \{(0,0,-1), (0,0,1)\} \rightarrow \complexes \setminus \{0\}$, and a conformal map from $\complexes \setminus \{0\}$ onto a sector isometric to $C(\alpha)$.

\begin{figure}
\includegraphics[bb = -1 0 406 304, height = 5cm]{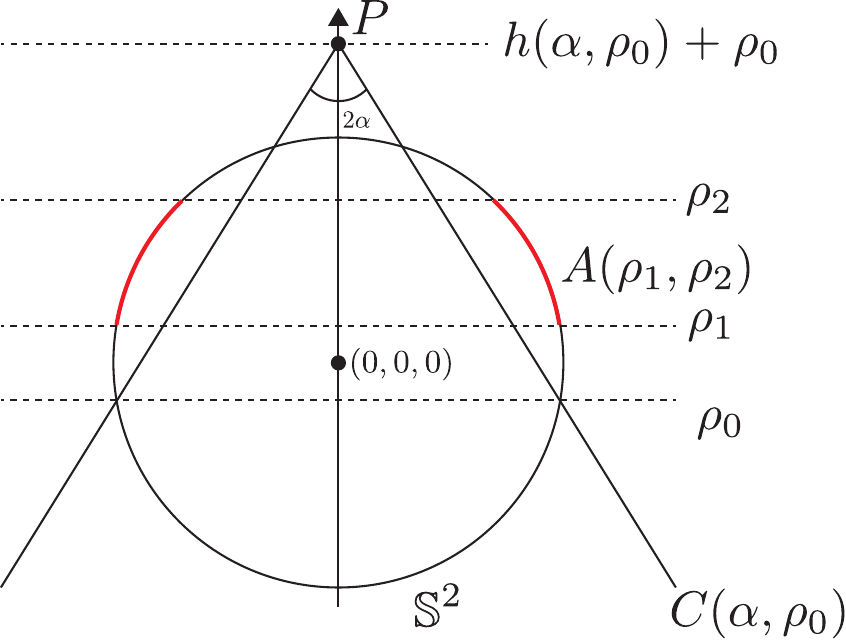}
\caption{The locations of the cone $C(\alpha,\rho_0)$ and the annulus $A(\rho_1,\rho_2)$. The cone $C(\alpha,\rho_0)$ has  apex angle $2\alpha$ and the lower circle in the intersection $\mathbb{S}^2\cap C(\alpha,\rho_0)$ is the parallel at $z=\rho_0$.}
\label{fig:cone}
\end{figure}
\begin{figure}

\begin{minipage}[b]{0.49\columnwidth}

\includegraphics[bb = 0 0 260 224, height = 5cm]{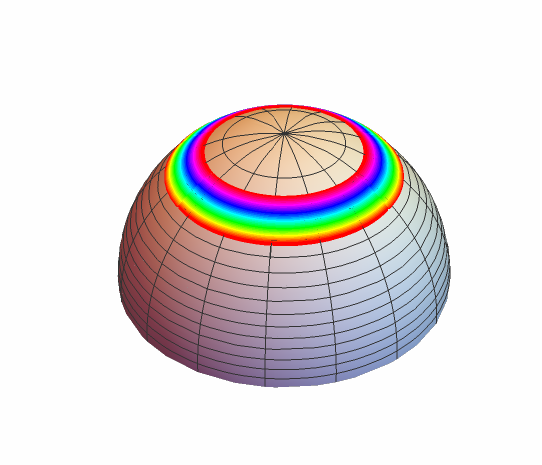}
\end{minipage}
\begin{minipage}[b]{0.49\columnwidth}
\includegraphics[bb = 0 0 260 204, height = 5cm]{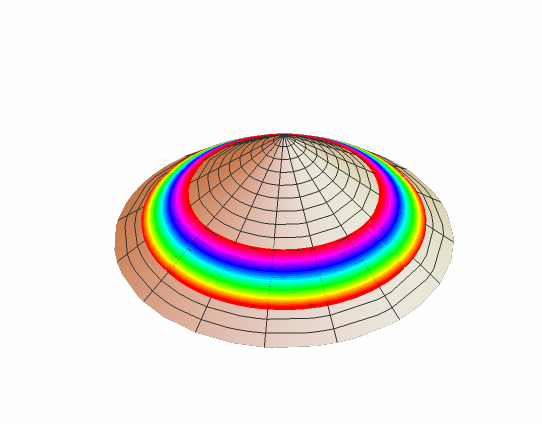}

\end{minipage}
\caption{The domain and the range of a Lambert conical projection}
\label{fig:cone1}
\end{figure}

\begin{figure}
\centering
\includegraphics[width=1\linewidth]{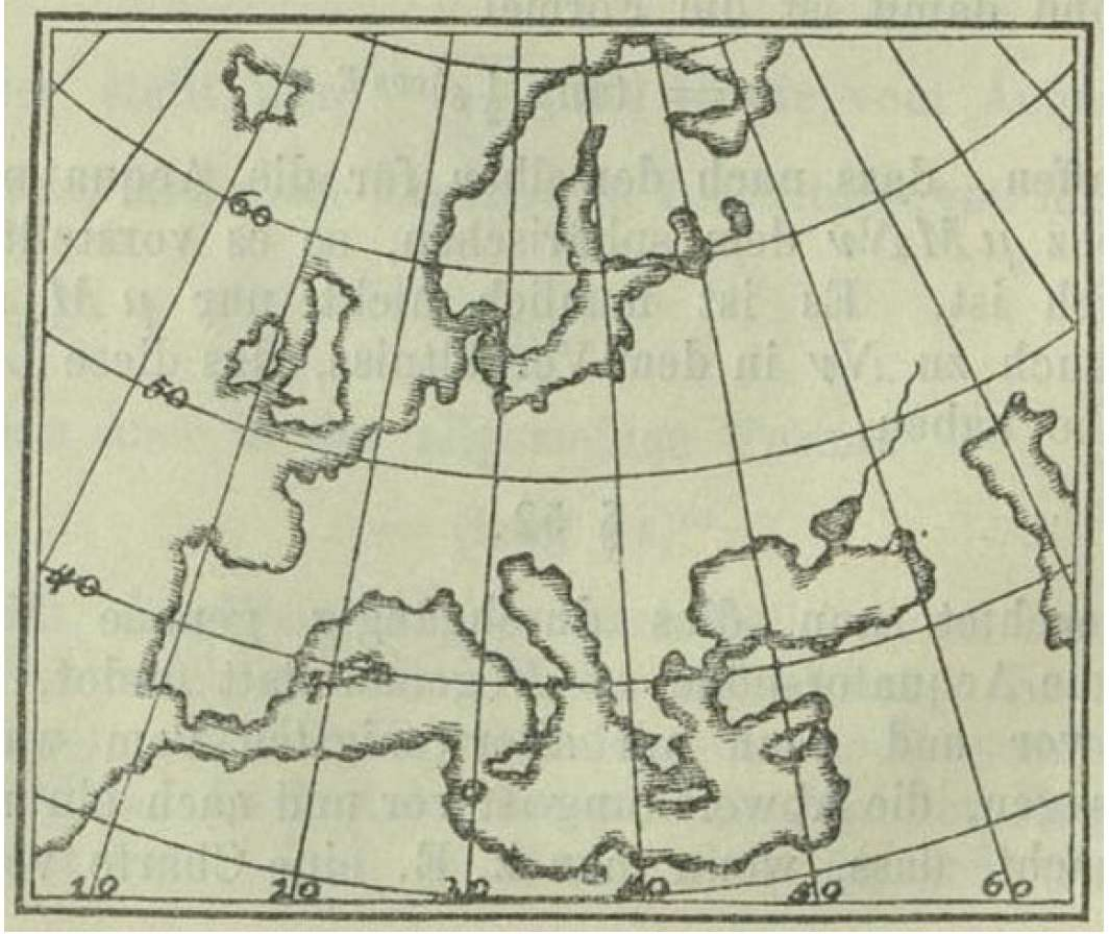}
\caption{A Lambert conical map of Europe, extracted from his memoir \emph{Beytr{\"a}ge zum Gebrauche der Mathematik und deren Anwendung} \cite{LamG}.}
  \label{fig:Lambert-map}
\end{figure}

\section{Stereographic projection and the modulus of a concentric spherical annulus}
We now start some calculations which will provide concrete mathematical expressions of the Lambert projection.
We shall first give a formula for the modulus of a spherical concentric annulus by considering its image under the polar stereographic projection.

Let us first set some notation.

We use the notation $(X,Y,Z)$ for the ordinary Cartesian coordinate system of $\mathbb{R}^3$.
The Euclidean metric on $\reals^3$ is defined by 
$
ds^2_{\mathbb{R}^3}=dX^2+dY^2+dZ^2.
$
The unit sphere in $\reals^3$  is expressed as $\mathbb{S}^2=\{(X,Y,Z)\in \mathbb{R}^3\mid X^2+Y^2+Z^2=1\}$.
We parametrise  as follows the unit sphere using two  parameters $\theta$ and $\rho$ corresponding to the longitude and the $Z$-coordinate:
$$
[0,2\pi]\times [-1,1]\ni (\theta,\rho)\mapsto (\sqrt{1-\rho^2}\cos\theta,\sqrt{1-\rho^2}\sin\theta,\rho)\in \mathbb{S}^2.
$$
 Under this parametrisation, the polar stereographic projection $\Pi$ with centre at $(0,0,1)$, which we call, for simplicity, stereographic projection,  is a conformal map from $\mathbb{S}^2\setminus \{(0,0,1)\}$ onto the $XY$-plane $\mathbb{C}=\mathbb{R}^2$ defined by
\begin{equation}
\label{eq:sphere_parameter}
\Pi(\theta,\rho)=w=u+iv=\sqrt{\frac{1+\rho}{1-\rho}}e^{i\theta}.
\end{equation}
 
Then $\Pi$ takes the south pole $(0,0,-1)\in \mathbb{S}^2$ to  the origin of the $XY$-plane.
The spherical metric on  $\mathbb{S}^2$ induced from $\reals^3$  is expressed in the $w$-plane, which is the range of $\Pi$, as
\[
(\Pi^{-1})^*ds^2_{\mathbb{R}^3}=\dfrac{4}{(1+|w|^2)^2}|dw|^2.
\]


Given two real numbers $\rho_1, \rho_2$ satisfying  $-1<\rho_1<\rho_2<1$,
the associated spherical concentric annulus is defined by 
$$A(\rho_1,\rho_2)=\{(\theta,\rho)\in \mathbb{S}^2\mid \rho_1<\rho<\rho_2\}.$$
The stereographic projection $\Pi$ sends  $A(\rho_1, \rho_2)$ to the planar concentric annulus
\[
\left\{
\sqrt{\frac{1+\rho_1}{1-\rho_1}}<|w|<\sqrt{\frac{1+\rho_2}{1-\rho_2}}
\right\}.
\]
Therefore, the modulus of the image of the annulus is given by the following formula.
\[
{\rm Mod}(A(\rho_1,\rho_2))=\frac{1}{2\pi}\log \sqrt{\frac{1-\rho_1}{1+\rho_1}}\sqrt{\frac{1+\rho_2}{1-\rho_2}}
=\frac{1}{4\pi}\log \frac{1-\rho_1}{1+\rho_1}\frac{1+\rho_2}{1-\rho_2}.
\]
This result was already known to Herbert Gr\"otzsch (1928). The reader is referred to  Gr\"otzsch's original article \cite{Grotzsch1}; see also its  translation in \cite{Grotzsch2}.

\section{Conformal maps from the complement of the poles in the sphere to a circular cone}
\subsection{The flat structure on a circular cone}
We first denote by $C(\alpha, A)$ the cone which is rotationally symmetric with respect to the $Z$-axis, oriented downwards, whose apex is located at $(0, 0, A)$, and whose apex angle is $\alpha$. 

For  $\alpha$ and $\rho_0$ satisfying  $0<\alpha<\pi/2$ and $-1<\rho_0<1$,
we impose the following condition:
\begin{equation}
(h(\alpha,\rho_0)+\rho_0)\sin\alpha<1, 
\label{eq:condition-rho0-1}\\
\end{equation}
where the function $h$ is defined by
\[
h(\alpha,\rho)=\frac{\sqrt{1-\rho^2}}{\tan\alpha}.
\]
This condition \eqref{eq:condition-rho0-1} says that the circular cone $C(\alpha,\rho_0 +  h(\alpha,\rho_0) )$ intersects the unit sphere $\mathbb{S}^2$ along two circles. 
Here, $\rho_0$ is the $Z$-coordinate of the lower small circle where $C(\alpha,\rho_0 + h(\alpha,\rho_0))$ and $\mathbb{S}^2$ intersect.

The cone $C(\alpha,\rho + h(\alpha,\rho_0))$ in $\mathbb{R}^3$ is identified as
\[
C(\alpha,\rho_0 + h(\alpha,\rho_0))=
\left\{(X,Y,Z)\in \mathbb{R}^3\mid Z=\rho+h(\alpha,\rho_0)-\frac{\sqrt{X^2+Y^2}}{\tan\alpha}.
\right\},
\]
\begin{equation}
\rho_0=\sin^2\alpha(h(\alpha,\rho_0)+\rho_0)
-\cos \alpha
\sqrt{1-(h(\alpha,\rho_0)+\rho_0)^2\sin^2\alpha}.
\label{eq:condition-rho0-2}
\end{equation}
and Condition \eqref{eq:condition-rho0-2} implies that that $Z$-coordinate of the lower intersection with $\mathbb{S}^2$ is equal to $\rho_0$.

Let $S(\alpha)=\{\zeta\in \mathbb{C}\mid 0<\arg \zeta \leq 2\pi\sin\alpha\}$ be the sector in the $\zeta$-plane $\mathbb{C}=\mathbb{R}^2$ corresponding to the developed image of $C(\alpha,\rho_0 + h(\alpha + \rho_0))$ in that plane.

 Using the polar coordinates $\zeta=re^{i\theta}$ on $S(\alpha)\subset \mathbb{C}$, we define the map $\Phi_0 \colon S(\alpha)\to C(\alpha,\rho_0 + h(\alpha + \rho_0))$ by
\begin{align*}
\Phi_0(r,\theta)
&=
(0,0,\rho_0+h(\alpha,\rho_0)) \\
&\quad+\left(
r\sin \alpha \cos\left(\frac{\theta}{\sin\alpha}\right),
r\sin \alpha \sin\left(\frac{\theta}{\sin\alpha}\right),
-r\cos\alpha
\right).
\end{align*}

The map $\Phi_0$ is an isometry with respect to the Euclidean metric on $S(\alpha)$ and the first fundamental form of $C_\alpha$. Indeed, we have
\begin{align*}
&(\Phi_0)^* ds^2_{\mathbb{R}^3}=
\left(
\sin \alpha\cos\left(\frac{\theta}{\sin\alpha}\right)dr
-r\sin \alpha \sin\left(\frac{\theta}{\sin\alpha}\right)d\theta
\right)^2 \\
&\quad+
\left(
\sin \alpha\sin\left(\frac{\theta}{\sin\alpha}\right)dr+r\sin \alpha \cos\left(\frac{\theta}{\sin\alpha}\right)d\theta
\right)^2+
\cos^2\alpha dr^2 \\
&=dr^2+r^2d\theta^2,
\end{align*}
where $\zeta=re^{i\theta} \in S(\alpha)$.

We now compose the following two maps, and construct a conformal map $\Phi_1$ from $\mathbb{C}\setminus \{0\}$ onto $C(\alpha, \rho_0 + h(\alpha, \rho))\setminus \{(0,0,\rho_0+h(\alpha,\rho_0))\}$:
$$
z \mapsto z^{\sin(\alpha)}=\zeta=re^{i\theta} \mapsto\Phi_0(r,\theta),
$$
where the domains and ranges are
$$
\mathbb{C}\setminus \{0\} \longrightarrow S(\alpha)\setminus \{0\} \longrightarrow C(\alpha,\rho_0+ h(\alpha, \rho_0))\setminus \{(0,0,\rho_0+h(\alpha, \rho_0))\}. 
$$

Using the polar coordinate $z=Re^{i\Theta}$ on the $z$-plane, the composition $\Phi_1\colon \mathbb{C}\setminus \{0\}\to C(\alpha,\rho_0+ h(\alpha, \rho_0))\setminus (0,0,\rho_0+h(\alpha,\rho_0))$ is expressed as
\begin{equation}
\label{eq:cone_parameter}
\Phi_1(R,\Theta)=
\left(
R^{\sin\alpha}\sin \alpha \cos\left(\Theta\right),
R^{\sin\alpha}\sin \alpha \sin\left(\Theta\right),
\rho_0+h(\alpha,\rho_0)-R^{\sin\alpha}\cos\alpha
\right)
\end{equation}
and the pull-back metric $\Phi_1^*ds^2_{\mathbb{R}^3}$ of the flat metric on $C(\alpha)\setminus (0,0,\rho_0+h(\alpha,\rho_0))$ becomes
\begin{equation}
\label{eq:Euclid_metric_Cone}
\Phi_1^*ds^2_{\mathbb{R}^3}=( R^{\sin\alpha-1}\sin \alpha)^2(dR^2+R^2d\Theta^2).
\end{equation}

\subsection{The normalised conformal representation of the sphere}
We set
\[
R_{\rho_0}=\left(\frac{\sqrt{1-\rho_0^2}}{\sin\alpha}\right)^{1/\sin\alpha}.
\]
Then, the circle $\{z\in \mathbb{C}\mid |z|=R_{\rho_0}\}$ in $\mathbb{C}$ is mapped by $\Phi_1$ to the lower component of the two small circles $C(\alpha,\rho_0+ h(\alpha, \rho_0)) \cap \mathbb{S}^2$, which we denote by  $\gamma(\rho_0)=\{Z=\rho_0\}\subset \mathbb{S}^2$ .
We define a map $\Phi_3: \mathbb{C}\setminus \{0\} \rightarrow \mathbb{C}\setminus \{0\}$ by 
\[
w=\Phi_3(z)=
R_{\rho_0}\sqrt{\frac{1+\rho_0}{1-\rho_0}}\frac{1}{z}
=\left(\frac{z}{R_{\rho_0}}\sqrt{\frac{1-\rho_0}{1+\rho_0}}\right)^{-1}.
\]
Since $1/\Pi(z)$ is a conformal representation of $\mathbb{S}^2\setminus \{(0,0,1), (0, 0, -1)\}$ taking the north pole to the origin (in the limiting sense), the composition
\[
\Phi_2=\Pi^{-1}\circ \Phi_3
\]
is a conformal parameterisation sending the circle $R=R_{\rho_0}$ in the $z$-plane
to the circle $\gamma(\rho_0)$ in $\mathbb{S}^2$. This means that $\Phi_1$ and $\Phi_2$ are normalised by the condition
\[
\Phi_1(\{|z|=R_{\rho_0})=\Phi_2(\{|z|=R_{\rho_0})=\gamma(\rho_0).
\]
By construction, the conformal maps $\Phi_1$ and $\Phi_2$ are equivariant with respect to the following actions of $\mathbb{S}^1$:
the action of $\mathbb{S}^1$ on the domain $\mathbb{C}$ of $\Phi_1$ and $\Phi_2$ , that is, the rotation around the origin, and the action of $\mathbb{S}^1$ on the range $\mathbb{R}^3$, that is, the rotation around the $Z$-axis.

The resulting map
$$
\mathbb{S}^2\setminus \{(0,0,-1), (0,0,1)\}\ni (X,Y,Z)\mapsto \Phi_1\circ \Phi_2^{-1}(X,Y,Z)\in C(\alpha, \rho_0) \setminus \{P\}
$$
is conformal and extends continuously over the north pole to a homeomorphism from $\mathbb{S}^2 \setminus \{(0,0,-1)\}$ to $C(\alpha,\rho_0)$ (see \Cref{fig:CD}). 
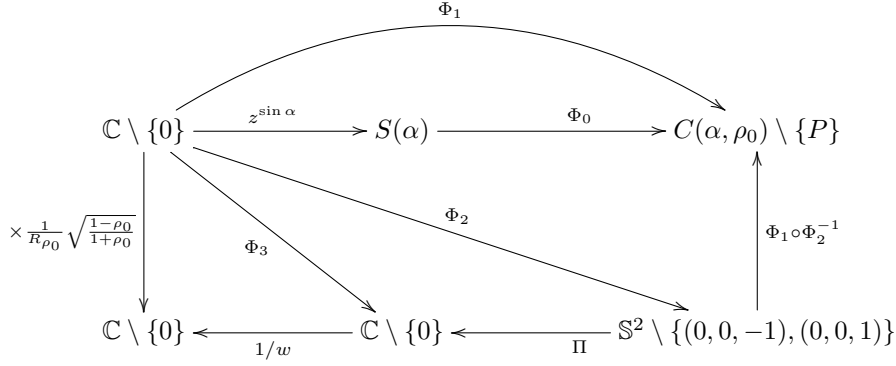
\begin{figure}[t]
\[
\xymatrix@C=60pt@R=60pt{
\mathbb{C} \setminus \{0\}
\ar[rd]_{\Phi_3}
\ar@/^40pt/[rr]^{\Phi_1}
\ar[r]^{z^{\sin\alpha}}
\ar[d]_{\times \frac{1}{R_{\rho_0}}\sqrt{\frac{1-\rho_0}{1+\rho_0}}} 
\ar[rrd]^{\Phi_2}
& S(\alpha) 
\ar[r]^{\Phi_0}
& C(\alpha,\rho_0) \setminus \{P\}\\
\mathbb{C} \setminus \{0\}
&
\mathbb{C}\setminus\{0\}  \ar[l]^{1/w}
& \mathbb{S}^2\setminus\{(0,0,-1), (0,0,1)\}
\ar[l]^{\Pi}
\ar[u]_{\Phi_1\circ \Phi_2^{-1}}
}
\]
\caption{ Conformal maps from the sphere $\mathbb{S}^2\setminus \{(0,0,-1),(0,0,1)\}$ to  $C(\alpha,\rho_0) \setminus \{P\}$. The continuous extension of the conformal map $\Phi_1\circ \Phi_2^{-1}$ takes the north pole to $P$.}
\label{fig:CD}
\end{figure}

\section{Lambert's calculation}
We now compare our calculation with Lambert's original calculation for his projections (see \cite{LamG} for Lambert's original article and its English translation, and \cite{A} for a recent exposition of this work.). 
Since the projections are equivariant with respect to the rotations, we need only to determine the heights of the images of the parallels, and hence the problem is to determine a function on a meridian describing the height of the image of each point there.
Lambert, based on his deep insight into the geometric behaviour of a conformal map (namely, the Lambert map), derived a differential equation for this, which describes the distance from the apex on the cone as a function of the angle coordinate on the sphere, see \cite[p. 195]{CP}.
Let $\epsilon$ denote the spherical distance from the north pole.  
From our conformal representation $\Phi_1$ of the cone, the distance from the cone is given by  
\begin{align*}
\mathrm{dist} &=
R_{\rho_1}^{\sin\alpha}
\left(
\frac{1+\rho_1}{1-\rho_1}\frac{1-\cos\epsilon}{1+\cos\epsilon}
\right)^{\sin \alpha/2} \\
&=
\frac{\sqrt{1-\rho_1^2}}{\sin\alpha}
\left(\frac{1+\rho_1}{1-\rho_1}\right)^{\sin \alpha/2}
\left(\tan\frac{\epsilon}{2}\right)^{\sin\alpha},
\end{align*}
a value which coincides with Lambert's calculation (up to a multiplicative constant).  

\begin{figure}
\includegraphics[bb =-1 0 475 245, width = 12cm]{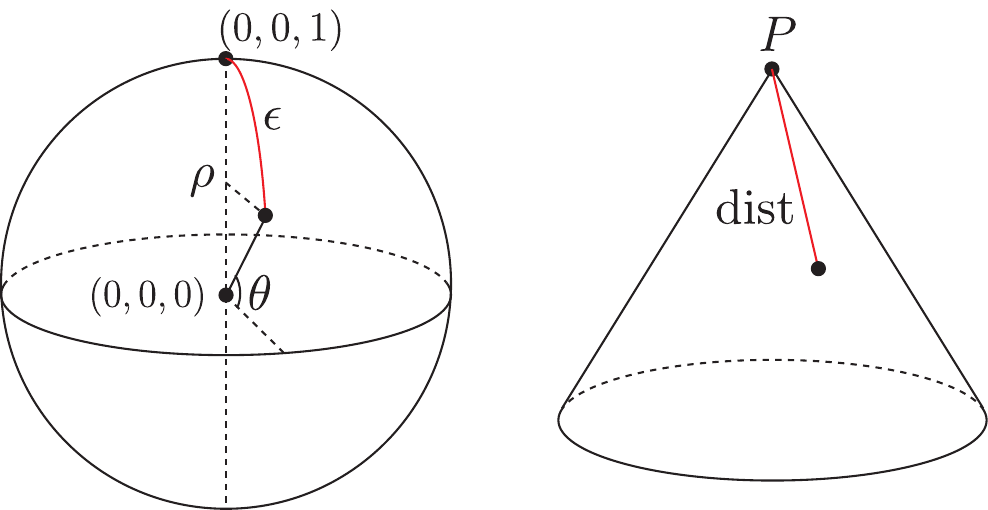}
\caption{Comparison between Lambert's calculation and ours. The $\epsilon$ indicates the spherical distance from the north pole, and ${\rm dist}$  the distance from the apex.}
\label{fig:Lambert_cal}
\end{figure}

\section{Comparison of the metrics}
The two induced metrics $\Phi_1^*ds^2_{\mathbb{R}^3}$ and $\Phi_2^*ds^2_{\mathbb{R}^3}$ defined on the $z$-plane are conformal to each other, and $z$ is the isothermal coordinate. From the rotation equivariance, the ratio of the two conformal $(0, 2)$ tensors, which is a function on $\mathbb{C}$,  
\[
\lambda =\frac{\Phi_1^*ds^2_{\mathbb{R}^3}}{\Phi_2^*ds^2_{\mathbb{R}^3}},
\]
depends only on the radial coordinate $R=|z|$ of the $z$-plane. Hence, we can  think of the ratio $\lambda$ as a function of the radial parameter $R$. 

Note that 
\begin{align*}
&\Phi_2^*ds^2_{\mathbb{R}^3}=
\left(
\dfrac{4}{\left(1+
\left|
R_{\rho_0}\sqrt{\frac{1+\rho_0}{1-\rho_0}}\frac{1}{z}
\right|^2
\right)^2}
\left|
\frac{d}{dz}
\left(
R_{\rho_0}\sqrt{\frac{1+\rho_0}{1-\rho_0}}\frac{1}{z}
\right)
\right|^2
\right)
|dz|^2 \\
&=
4R^2_{\rho_0}\left(\frac{1+\rho_0}{1-\rho_0}\right)
\left(
R_{\rho_0}^2\frac{1+\rho_0}{1-\rho_0}+|z|^2
\right)^{-2}|dz|^2.
\end{align*}
Recall  the explicit expression of  $\Phi_1^*ds^2_{\mathbb{R}^3}$ as in \eqref{eq:Euclid_metric_Cone}.

We can now calculate $\lambda$ explicitly:
\begin{align}
\lambda(R)
&=
(R^{\sin\alpha-1}\sin\alpha)^2
\frac{1}{4R^2_{\rho_0}}\frac{1+\rho_0}{1-\rho_0}
\left(
R_{\rho_0}^2\frac{1-\rho_0}{1+\rho_0}+R^2
\right)^{2}.
\label{eq:ratio_R}
\end{align}

For $R>0$, the height $\rho$ of the small circle $\Phi_2(\{|z|=R\})$ on $\mathbb{S}^2$ satisfies
\[
R=R_{\rho_0}\sqrt{\frac{1-\rho}{1+\rho}}
\sqrt{\frac{1+\rho_0}{1-\rho_0}}.
\]
Therefore, the ratio \eqref{eq:ratio_R} in terms of $\rho$ is expressed as
\begin{align*}
\Lambda(\rho)
&=\Lambda(\rho,\alpha,\rho_0)
\\
&=\lambda\left(R_{\rho_0}\sqrt{\frac{1-\rho}{1+\rho}}\sqrt{\frac{1+\rho_0}{1-\rho_0}}\right) \\
&=
\frac{\sin^2\alpha}{4}
\dfrac{1}{R_{\rho_0}^2}\dfrac{1-\rho_0}{1+\rho_0}
\left(
R_{\rho_0}^2\dfrac{1+\rho_0}{1-\rho_0}+
\left(R_{\rho_0}\sqrt{\frac{1-\rho}{1+\rho}}\sqrt{\frac{1+\rho_0}{1-\rho_0}}\right)^2
\right)^2
\\
&\qquad \times \left(R_{\rho_0}\sqrt{\frac{1-\rho}{1+\rho}}\sqrt{\frac{1+\rho_0}{1-\rho_0}}\right)^{2\sin\alpha-2}
\\
&=\frac{\sin^2\alpha}{4}
\left(R_{\rho_0}^2\dfrac{1+\rho_0}{1-\rho_0}
\right)^{\sin\alpha}\left(1+\frac{1}{1+\rho}\right)^2
\left(\frac{1-\rho}{1+\rho}\right)^{\sin\alpha-1} \\
&=\sin^2\alpha
\left(R_{\rho_0}^2\dfrac{1+\rho_0}{1-\rho_0}
\right)^{\sin\alpha}\frac{1}{(1-\rho)^{1-\sin\alpha}(1+\rho)^{1+\sin\alpha}}
\\
&=\sin^2\alpha
\frac{1-\rho_0^2}{\sin^2\alpha}
\left(\dfrac{1+\rho_0}{1-\rho_0}
\right)^{\sin\alpha}\frac{1}{(1-\rho)^{1-\sin\alpha}(1+\rho)^{1+\sin\alpha}}
\\
&=
\frac{(1-\rho_0)^{1-\sin\alpha}(1+\rho_0)^{1+\sin\alpha}}{(1-\rho)^{1-\sin\alpha}(1+\rho)^{1+\sin\alpha}}.
\end{align*}
We define $L(\rho, \alpha, \rho_0)$ to be $\sqrt{\Lambda(\rho, \alpha, \rho_0)}$.
We have thus proved the following.

\begin{proposition}
The number
 \[
 L(\rho, \alpha, \rho_0) = \sqrt{\frac{(1-\rho_0)^{1-\sin\alpha}(1+\rho_0)^{1+\sin\alpha}}{(1-\rho)^{1-\sin\alpha}(1+\rho)^{1+\sin\alpha}}}
 \]
 is the infinitesimal Lipschitz constant 
of the normalised conformal map
$$
\Phi_1\circ \Phi_2^{-1}\colon \mathbb{S}^2\setminus\{(0,0,-1), (0,0,1)\}\to C(\alpha,\rho_0 + h(\alpha, \rho_0)) \setminus \{(0,0,\rho_0+h(\alpha,\rho_0)\}.
$$
\end{proposition}

\section{Optimal apex angle $\alpha$ for the projection}
A bi-Lipschitz map $f\colon X \to Y$ between metric spaces $(X,d_X)$ and $(Y, d_Y)$ is necessarily a homeomorphism. The distortion of a bi-Lipschitz map $f$ is defined to be 
$$
\log\left(\sup_{x_1 \neq x_2} \frac{d_Y(f(x_1), f(x_2))}{d_X(x_1, x_2)}\right)-\log\left(\inf_{x_1 \neq x_2}  \frac{d_Y(f(x_1), f(x_2))}{d_X(x_1, x_2)}\right).
$$
For a spherical annulus $A(\rho_1,\rho_2)$ with the conditions  $0<\alpha<\pi/2$ and $-1<\rho_0<1$, we define $\delta(\rho_1,\rho_2,\alpha,\rho_0)$ by
\[
\delta(\rho_1,\rho_2,\alpha,\rho_0)=
\log
\left\{
\frac{\sup_{\rho_1<\rho<\rho_2}L(\rho,\alpha,\rho_0)}{\inf_{\rho_1<\rho<\rho_2}L(\rho,\alpha,\rho_0)}
\right\}.
\]
Since $\Phi_1 \circ \Phi_2^{-1}$ is conformal, $\delta(\rho_1,\rho_2,\alpha,\rho_0)$ is equal to the distortion of this map. This follows from the observation that 
any path $\sigma$ connecting $f(x_1)$ and $f(x_2)$ has length shorter than that of the inverse image $f^{-1}(\sigma)$ of the path times the constant $\sup_{\rho_1<\rho<\rho_2}L(\rho,\alpha,\rho_0)$, and that $\sigma$ has length longer than that of the inverse image $f^{-1}(\sigma)$ of the path times the constant $\inf_{\rho_1<\rho<\rho_2}L(\rho,\alpha,\rho_0)$.

In this section, our aim is to choose an optimal apex angle $\alpha=\alpha_0$ such that some conditions which we shall state now are satisfied, assuming that $\rho_0=\rho_1$.
There are two different  conditions which we can naturally impose on the shape of the cone, in order to get an optimal apex angle $\alpha_0$:
\begin{enumerate}[(a)]
\item
The distortion $\delta(\rho_0,\rho_2,\alpha_0,\rho_0)$ is minimal among all $0<\alpha<\pi/2$ and $\rho_0$ which satisfy  the conditions \eqref{eq:condition-rho0-1} and \eqref{eq:condition-rho0-2}.
\item
$L(\rho_0,\alpha_0,\rho_0)=L(\rho_2,\alpha_0,\rho_0)=1$.
\end{enumerate}
Historically, the second condition was used after choosing two parallels passing through cities regarded as important for the purpose of the map. This is the case for the Delisle--Euler and Marinus--Ptolemy maps, see the introduction of the present paper and the paper \cite{Geography-GB-OMP} for a review.

In fact these two conditions are equivalent.

\begin{theorem}
\label{optimal is isometric}
Given  $-1<\rho_1<\rho_2<1$,
the distortion $\delta$ attains a minimum at $\alpha_0$ if and only if  $L(\rho_2,\alpha_0,\rho_1)=1$.
Furthermore,  if this holds, then we have  $\rho_1<\sin(\alpha_0)<\rho_2$.
\end{theorem}

%

Before proving \Cref{optimal is isometric}, we give a proposition which describes the property of the function $\Lambda=L^2$.
Fix $-1<\rho_1<\rho_2<1$.
We define the function $F$ by 
$$
F(x,y,z)=\frac{(1+z)^{1+y}(1-z)^{1-y}}{(1+x)^{1+y}(1-x)^{1-y}}
$$
for $-1<x$, $y$, $z<1$. We note that
$$
\Lambda(\rho_2,\alpha,\rho_1)=F(\rho_2,\sin(\alpha),\rho_1).
$$

\begin{proposition}
\label{prop:1}
The following properties hold:
\begin{itemize}
\item[{\rm (a)}]
$F(x,y,x)=1$;
\item[{\rm (b)}]
$F(z,y,x)=1/F(x,y,z)$;
\item[{\rm (c)}]
$\dfrac{F(x,y,z)}{F(y,y,z)}=F(x,y,y)$;
\item[{\rm (d)}]
If we fix $x>-1$ and $z<1$ with $x>z$, then the function $F(x,y,z)$ is strictly monotone decreasing with respect to  $y\in (-1,1)$.
\item[{\rm (e)}]
If we fix $y>-1$ and $z<1$, then the function $F(x,y,z)$ is strictly convex with respect to  $x \in (-1,1)$, and  $F(x,y,z)\to \infty$ as $x\to \pm 1$. Moreover, the function $F(x,y,z)$ attains its minimum at $x=y$.
\item[{\rm (f)}]
If we fix $z \in (-1,1)$, then the function $F(x,x,z)$ is strictly monotone increasing with respect to $x$ in $(-1,z)$ and is strictly monotone decreasing in $(z,1)$.
\end{itemize}
\end{proposition}

\begin{proof}
The statements (a), (b), (c) are trivial.
To prove (d), we observe
\begin{align*}
\frac{\partial}{\partial y}F(x,y,z)
=\left(\log\frac{1+z}{1+x}+\log\frac{1-x}{1-z}\right)F(x,y,z).
\end{align*}
If $x>z$, the derivative of $\log F(x,y,z)$ in $y$ is negative for $y\in (-1,1)$. Hence, $y\mapsto F(x,y,z)$ is strictly monotone decreasing.

We next prove (e). The second derivative of $F(x,y,z)$ in $x$ is
\begin{align*}
2\frac{(1+z)^{1+y}(1-z)^{1-z}(1+3x^2-6xy+2y^2) }{(1+x)^{y+3}(1-x)^{y-3}}\\
=2\frac{(1+z)^{1+y}(1-z)^{1-z}}{(1+x)^{y+3}(1-x)^{y-3}}(3(x-y)^2+1-y^2)
\end{align*}
which is positive if $-1<y<1$.
Therefore, $F(x,y,z)$ is strictly convex with respect to $x$ in $(-1,1)$, and we also have  $F(x,y,z)\to \infty$ as $x\to \pm 1$. It follows that with $y$ and $z$ fixed,  the  function $x\mapsto F(x,y,z)$ takes a unique minimum at some $x_0\in (-1,1)$.
Since the first derivative of $F(x,y,z)$ in $x$ is 
\[
2\frac{(1+z)^{1+y}(1-z)^{1-y}(x-y)}{(1+x)^{y+2}(1-x)^{y-2}},
\]
we see that  $x_0=y$.

We finally show (f).
We can easily see that
$$
\dfrac{\partial}{\partial x}\log F(x,x,z)=\log \frac{1-x}{1-z}+\log \frac{1+z}{1+z},
$$
hence $F(x,x,z)$ is strictly monotone increasing over $(-1,z)$ and strictly monotone decreasing over $(z,1)$.
\end{proof}
%

\begin{proof}[Proof of \Cref{optimal is isometric}]
We use $F(x,y,z)$ instead of  $\Lambda$ to make the description simpler.
We first observe that there exists a unique $a_0$ with $\rho_1<a_0<\rho_2$ such that $F(\rho_2,a_0,\rho_1)=1$. 
Indeed since
\begin{align*}
\lim_{a\to -1}F(\rho_2,a,\rho_1)
&=\left(\frac{1-\rho_1}{1-\rho_2}\right)^2>1 \\
\mathrm{and}
\\
\lim_{a\to 1}F(\rho_2,a,\rho_1)
&=\left(\frac{1+\rho_1}{1+\rho_2}\right)^2<1,
\end{align*}
by (d) of Proposition \ref{prop:1}, there exists a unique $-1<a_0<1$ such that $F(\rho_2,a_0,\rho_1)=1$. 
On the other hand,  by (a) of Proposition \ref{prop:1}, $F(\rho_1,a_0,\rho_1)=1=F(\rho_2,a_0,\rho_1)$, and by (e) of Proposition \ref{prop:1}, the map $\rho\mapsto F(\rho,a_0,\rho_1)$ is strictly convex, and the function $\rho\mapsto F(\rho,a_0,\rho_1)$ attains its minimum at $\rho=a_0$.
Therefore we have $\rho_1<a_0<\rho_2$.

We define the two functions $M$ and $m$  of $a$ by 
\begin{align*}
M(a)&=\sup_{\rho_1<\rho<\rho_2}F(\rho,a,\rho_1), \\
m(a)&=\inf_{\rho_1<\rho<\rho_2}F(\rho,a,\rho_1).
\end{align*}
Since $F$ is strictly convex with respect to the first coordinate, it attains $M(a)$  at $\rho=\rho_1$ or $\rho_2$. 
By (d) in Proposition \ref{prop:1}, we have $F(\rho_2,a,\rho_1)>1$ if $-1<a\le a_0$, and $F(\rho_2,a,\rho_1)<1$ if $a_0<a<1$. 
Therefore, 
$M(a)=F(\rho_2,a,\rho_1)$ if $-1<a\le a_0$, and $M(a)=1$ if $a_0<a<1$.
Since the minimum of $F(\cdot, a,\rho_1)$ with respect to the first coordinate is attained at $\rho=a$,  it attains $m(a)$  at $\rho=\rho_2$ if $\rho_2<a$, $\rho=a$ if $\rho_1\le a\le \rho_2$ and $\rho=\rho_1$ if $a<\rho_1$.
Therefore, we have
\begin{align*}
\frac{M(a)}{m(a)}&=
\begin{cases}
F(\rho_2,a,\rho_1) & (-1<a<\rho_1) \\
\frac{F(\rho_2,a,\rho_1)}{F(a,a,\rho_1)}=F(\rho_2,a,a) & (\rho_1\le a<a_0) \\
1/F(a,a,\rho_1)& (a_0\le a<\rho_2) \\
1/F(\rho_2,a,\rho_1) & (\rho_2\le a<1).
\end{cases}
\end{align*}
By (f) in Proposition \ref{prop:1}, $F(\rho_2,a,a)=1/F(a,a,\rho_2)$ is strictly monotone decreasing with respect to $a\in(-1,\rho_2)$, and $F(a,a,\rho_1)$ is strictly monotone decreasing if $\rho_1<a$.
It follows that  that the quotient $M(a)/m(a)$ is strictly monotone decreasing over $(-1,a_0)$, and strictly monotone increasing over $(a_0,1)$. Thus we see that the quotient attains a unique minimum at $a=a_0$.
\end{proof}

\section{Infinitesimal bi-Lipschitz constants}
Suppose that $\rho_1$ and $\rho_2$ with $-1<\rho_1<\rho_2<1$ are given.
We take $\alpha_0$ satisfying $\Lambda(\rho_2,\alpha_0,\rho_1)=1$ which we gave in the preceding section.
In this section, we  study the infinitesimal bi-Lipschitz constant $\sigma(\rho)$ defined by
$$
\sigma(\rho)=\max
\left\{
L(\rho,\alpha_0,\rho_1),\frac{1}{L(\rho,\alpha_0,\rho_1)}
\right\}
$$
for $\rho \in [\rho_1, \rho_2]$. 
Since $\Lambda(\rho_1,\alpha_0,\rho_1)=\Lambda(\rho_2,\alpha_0,\rho_1)=1$ and $\Lambda(\rho,\alpha_0,\rho_1)$ is strictly convex in terms of the first coordinate,
$\Lambda(\rho,\alpha_0,\rho_1)<1$ for $\rho_1<\rho<\rho_2$.
Therefore,
$$
\sigma(\rho)=\frac{1}{L(\rho,\alpha_0,\rho_1)}
$$
for $\rho_1<\rho<\rho_2$ (see Figure \ref{fig:BiLip}).
\begin{figure}
\includegraphics[bb = 0 2 252 143, height = 4cm]{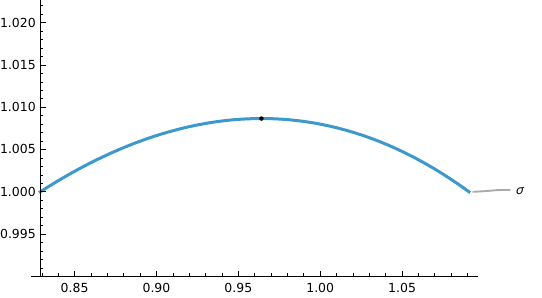}
\caption{Graph of the infinitesimal bi-Lipschitz constant $\sigma(\rho)$ on $\rho_1<\rho<\rho_2$}
\label{fig:BiLip}
\end{figure}

\section{Comparison with other maps of the Russian Empire that use conical projections}
In the paper \cite{MO}, the authors compared the distortions of several conical projections used for the construction of maps of the Russian\index{geographical map!of the Russian Empire (Delisle--Euler)} Empire.\index{General map of the Russian Empire (Delisle--Euler)} As we mentioned in the introduction of this paper, we chose this region of the world because this is the one whose map Euler famously drew during his stay in Saint Petersburg, see Figure \ref{fig:Empire}. The method Euler used to draw it is known under the Delisle--Euler projection. The projection was first used on an empirical basis by the famous French-Saint Petersburg geographer Joseph-Nicolas Delisle, and its theory, using differential calculus, was later developed by Euler, who showed that it was indeed appropriate to draw maps of regions like the Russian Empire whose span in longitude is very large.\footnote{The Russian Empire, spanning eleven time zones, included, at Euler's time, most of northern Eurasia.}
Among the maps considered in  \cite{MO}, the Delisle--Euler projection was shown to be the best one from several precise points of view.
The Lambert conformal conical projections was not included in the paper \cite{MO}.
In this section, we compare the distortion of the Lambert conical projection with that of the other conical projections used in \cite{MO}. We shall show that besides being conformal, the Lambert map is better than the other projections from various points of view.

We use the notation established in the previous sections. In the following, we consider the case when $\rho_0=\rho_1=\sin(47.30')=0.737277$ and $\rho_2=\sin(67.30')=0.92388$, corresponding to the southernmost point and the northernmost point of the Russian Empire.

\begin{figure}
\includegraphics[bb =0 2 259 160, height = 5cm]{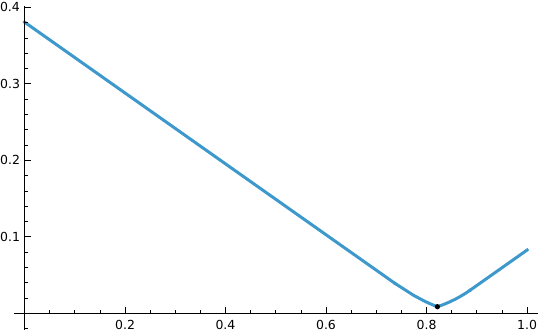}
\caption{The graph of the function $\delta(\rho_1,\rho_2,\arcsin(a),\rho_0)$ on $0<a<1$. The minimum attains at $a=0.821529=\sin(55.24')$.}
\label{fig:delta_Euler}
\end{figure} 
The graph of the distortion $\delta$ is given in \Cref{fig:delta_Euler}.
The distortion $\delta$ attains its minimum at $\alpha_0=0.9640=55\degree14'$
and the minimum value is $0.0086263354$.
By numerical computation, we see that the cone $C(\alpha_1,\rho_0+ h(\alpha_1, \rho_0))$ intersects the sphere $\mathbb{S}^2$ at  the lower circle of $\partial A(\rho_1,\rho_2)$ and the circle of the height $\rho=0.890819$, which lies inside $A(\rho_1,\rho_2)$ (see \Cref{fig:graph_delta}).
\begin{figure}
\includegraphics[bb = 0 2 244 137, height = 5cm]{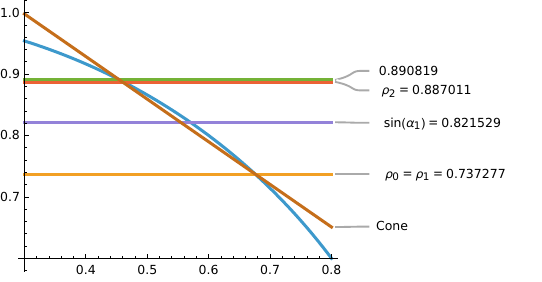}
\caption{Slices of the cone and the sphere along the $xz$-plane: The sphere and the cone interesect the circles of heights $\rho_1=0.737277$ and $\rho=0.890819>\rho_2=0.887011$.}
\label{fig:graph_delta}
\end{figure}

We now draw a graph comparing this Lambert projection with other conical projections which we studied in \cite{MO}.
We recall that there are five maps there: the projection from the centre, the horizontal projection, the Delisle--Euler projection, the Teichm\"{u}ller map, and the orthogonal projection.
Among these, it turned out that the horizontal projection has much worse distortion than the others, and we eliminate it from our study.

We now make some remark on the Delisle map which we chose in \cite{MO} and give another possible choice.
Delisle's method gives a family of projections of spherical annulus to an annulus on a circular cone taking the meridians into straight rays and the parallels to concentric circles around the apex in such a way the map restricted to the meridians  is a similarity map having a constant scaling factor common to all meridians.
In \cite{MO}, we chose a Delistle  projection which takes a spherical annulus bounded by two parallels $L_1$ and $L_2$ to an annulus bounded also by $L_1$ and $L_2$ on a circular cone intersecting the sphere along $L_1$ and $L_2$.
There is another natural choice: we fix a spherical annulus bounded by $L_1$ and $L_2$ and the same circular cone intersecting along $L_1$ and $L_2$, and then choose a Delisle projection which sends every meridian isometrically into a straight ray.
We take both of them into consideration.
We call the latter Delisle's equidistant conical projection.
\begin{figure}
\includegraphics[bb = 0 5 259 211, width = 12cm]{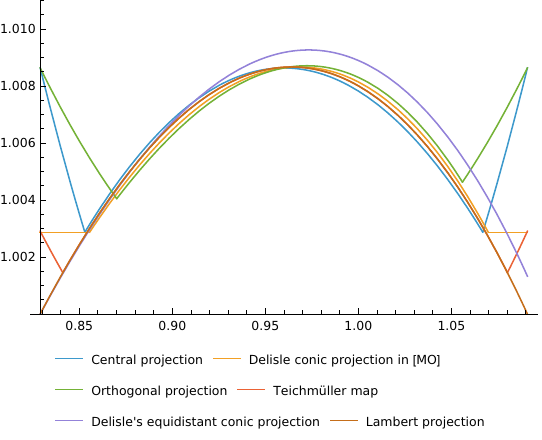}
\caption{Comparison of six maps.}
\label{fig:comparison}
\end{figure}
We now draw a graph of infinitesimal bi-Lipschitz constants regarded as functions of $\rho$ for six maps mentioned above (\Cref{fig:comparison}).
The graph shows that our Lambert map (with the choice of the best apex angle) has quite good values of infinitesimal bi-Lipschitz constants.
Numerically, the distortions of the six projections, the central projection, the Delisle projection (as in [MO]), Delisle's equidistant conical projection, the orthogonal projection, the Teichm\"{u}ller map, and the Lambert projection are 0.0171839, 0.00862621,  0.00921812, 0.00866925,  0.0115244 and 0.00862633 respectively.
This suggests that the Lambert projection with the best choice of $\alpha_0$ has a comparable distortion (although a bit worse than the Delisle projection chosen in \cite{MO}, better than the others).

As for the comparison between the Teichm\"{u}ller map and the Lambert projection, as shown in Figure~\ref{fig:comparison}, the graphs of their infinitesimal bi-Lipschitz constants almost coincide. 
Based on numerical calculations, the moduli of $A(\rho_1,\rho_2)$ and $B$ are ${\rm Mod}(A(\rho_1,\rho_2)) = 0.0737271$ and ${\rm Mod}(B) = 0.0739411$, respectively, and the maximal dilatation of the Teichm\"uller map is  equal to $1.0029\ldots = {\rm Mod}(B)/{\rm Mod}(A(\rho_1,\rho_2))$.  
Thus, the Teichm\"uller map is almost conformal and is expected to be close to a conformal map, which in our case is the Lambert map.  
Moreover, according to the numerical calculations, at $\alpha_0 = \sin^{-1}(0.821529442901464)$, where the bi-Lipschitz constant for the Lambert map attains its maximum, the value agrees with that of Teichm\"{u}ller map up to  (at least) 14 decimal places.

\section{Final remarks}

\noindent {\bf 1.} Considering the facts that the Lambert projection is conformal while the others are not, and that its distortion (in several senses) is comparable to that of the Delisle--Euler map, the Lambert projection can be regarded as a good choice for drawing a map of the Russian Empire.

\medskip

\noindent {\bf 2.} In our computations, we made use of  an explicit expression for the modulus of the planar concentric annulus which allowed us to estimate the distorsion of the conformal representation. Let us emphasize the fact that except in a few cases,  there is no explicit expression for the conformal representation of a (planar or non-planar) region and computing the metric distortion of the conformal representation of a general open simply connected region of the plane (or of the sphere) is an extremely difficult problem. We note in this respect that the question of the Lipschitz distortion of the conformal representation of a topological disc was already addressed by M. A. Lavrentieff, one of the three founders of the theory of quasiconformal mappings. The question is contained in his paper \emph{On a class of continuous representations}, translated into English, with commentaries, in \cite{Lavrentieff}; see in particular, p. 434 of the English translation.

\medskip 
\noindent {\bf 3.} We end with a few words about Lambert.\index{Lambert, Johann Heinrich}
Euler and Lambert had a strong relationship. Lambert was about 20 years younger than Euler\index{Euler, Leonhard} but he died five years before him. Lambert was a precursor of hyperbolic geometry, maybe the most important precursor. As many other mathematicians before him (including Euler), he tried to show that Euclid's parallel  postulate follows from the others. To do this, he tried to obtain a contradiction by developing a geometry which assumes Euclid's postulates and axioms except the parallel postulate which he replaced by his negation, and he sought a contradiction. He did not reach any contradiction, and the results he obtained are theorems of hyperbolic geometry. See the critical edition with a French translation of his memoir on this subject, titled \emph{The theory of parallel lignes} \cite{2012-Lambert}, and the extended review of that memoir in \cite{Papa-Theret-Lambert2}. For an exposition of Lambert's work on geographical maps, see \cite{A}. For a biography and a review of the general work of Johann Heinrich Lambert, the reader is referred to \cite{Papa-Lambert}.

\printindex

\end{document}